\documentclass[12pt]{amsart}
\usepackage[all]{xypic}
\usepackage{tikz}
\usetikzlibrary{arrows}
\usetikzlibrary{decorations.markings}

\usepackage{graphicx}
\usepackage{bm}
\usepackage{epsf}
\usepackage{verbatim} 
\usepackage{amsmath}
\usepackage{amsfonts}
\usepackage{amssymb}
\usepackage{mathrsfs}
\usepackage{amsthm}
\usepackage{newlfont}
 \newtheorem{thm}{Theorem}
 
 \newtheorem{cor}[thm]{Corollary}

 \newtheorem{lem}[thm]{Lemma}
 
 \newtheorem{prop}[thm]{Proposition}
 
 \theoremstyle{definition}

 \theoremstyle{definition}

  \theoremstyle{definition}

 \theoremstyle{remark}
 \newtheorem{rem}[thm]{Remark}
 
 \theoremstyle{remark}

 \theoremstyle{definition}
 \newtheorem{example}[thm]{Example}

\numberwithin{thm}{section}
\numberwithin{equation}{section}

 \newcommand{\Frob}{\mathrm{Frob}}

 \newcommand{\Aut}{\mathrm{Aut}}
 
 \newcommand{\End}{\mathrm{End}}

 \newcommand{\ord}{\mathrm{ord}}

 \newcommand{\Gal}{\mathrm{Gal}}
 \newcommand{\GL}{\mathrm{GL}}

 \newcommand{\sep}{\mathrm{sep}}
 \newcommand{\alg}{\mathrm{alg}}

 \newcommand{\rank}{\mathrm{rank}}

 \renewcommand{\mod}{\mathrm{mod}}

 \newcommand{\fl}{\mathfrak l}
 \newcommand{\fb}{\mathfrak b}
 
 \newcommand{\fp}{\mathfrak p}
 \newcommand{\fq}{\mathfrak q}
 
 \newcommand{\fn}{\mathfrak n}
 \newcommand{\fm}{\mathfrak m}
 \newcommand{\fd}{\mathfrak d}
 \newcommand{\fc}{\mathfrak c}

\newcommand{\cZ}{\mathcal{Z}}
 \newcommand{\cO}{\mathcal{O}}

 \newcommand{\cP}{\mathcal{P}}
 
 \newcommand{\cE}{\mathcal{E}}

 \renewcommand{\cD}{\mathcal{D}}

 \newcommand{\C}{\mathbb{C}}
 \newcommand{\F}{\mathbb{F}}
 \newcommand{\Q}{\mathbb{Q}}
 
 \newcommand{\Z}{\mathbb{Z}}

 \newcommand{\Nr}{\mathrm{Nr}}

 \newcommand{\eps}{\varepsilon}
 
 \newcommand{\To}{\longrightarrow}
 
 \newcommand{\Fi}{F_\infty}

 \newcommand{\La}{\Lambda}

\begin{document}

\title[Endomorphism rings of reductions of Drinfeld modules]{Endomorphism rings of reductions of Drinfeld modules}

\author{Sumita Garai}
\address{
	Department of Mathematics, 
	Pennsylvania State University,
	University Park, PA 16802, USA
} 
\email{sxg386@psu.edu}

\author{Mihran Papikian}
\address{
Department of Mathematics, 
Pennsylvania State University,
University Park, PA 16802, USA
} 
\email{papikian@psu.edu}

%\thanks{The author's research was partially supported by grants from the Simons Foundation (245676) and the National Security Agency 
%(H98230-15-1-0008).} 

\dedicatory{In memory of David Goss}



\begin{abstract}  Let $A=\F_q[T]$ be the polynomial ring over $\F_q$, and 
$F$ be the field of fractions of $A$. Let $\phi$ be a Drinfeld $A$-module of rank $r\geq 2$ over $F$. 
For all but finitely many primes $\fp\lhd A$, one can reduce $\phi$ modulo $\fp$ 
to obtain a Drinfeld $A$-module $\phi\otimes\F_\fp$ of rank $r$ over $\F_\fp=A/\fp$. The 
endomorphism ring $\cE_\fp=\End_{\F_\fp}(\phi\otimes\F_\fp)$ is an order in an imaginary 
field extension $K$ of $F$ of degree $r$. Let $\cO_\fp$ be the integral closure of $A$ in $K$, 
and let $\pi_\fp\in \cE_\fp$ be the Frobenius endomorphism of $\phi\otimes\F_\fp$. 
Then we have the inclusion of orders $A[\pi_\fp]\subset \cE_\fp\subset \cO_\fp$ in $K$. 
We prove that if $\End_{F^\alg}(\phi)=A$, then for arbitrary non-zero ideals $\fn, \fm$ of $A$
there are infinitely many $\fp$ 
such that $\fn$ divides the index $\chi(\cE_\fp/A[\pi_\fp])$ and $\fm$ divides the index $\chi(\cO_\fp/\cE_\fp)$.  
We show that the index $\chi(\cE_\fp/A[\pi_\fp])$ is related to a reciprocity law for 
the extensions of $F$ arising from the division points of $\phi$. In the rank $r=2$ case 
we describe an algorithm for computing the orders $A[\pi_\fp]\subset \cE_\fp\subset \cO_\fp$, and give some 
computational data. 
\end{abstract}

\maketitle

%---------------------------------------------

\section{Introduction} 

Let $\F_q$ be a finite field with $q$ elements and of characteristic $p$. 
Let $A=\F_q[T]$ be the polynomial ring over $\F_q$ in an indeterminate $T$, and 
$F=\F_q(T)$ be the field of fractions of $A$. In this paper we study the 
endomorphism rings of the reductions of a fixed Drinfeld module defined over $F$. 
We are interested in theoretical, as well as computational, aspects of the theory of these rings. 
To state the main results of the paper, we first need to introduce some notation and terminology. 

\subsection{Notation and terminology}  
The degree $\deg(a)$ of $0\neq a\in A$ is its degree as a polynomial in $T$. The degree function extends to a valuation of $F$; 
the corresponding place of $F$ is denoted by $\infty$. Let $\Fi=\F_q(\!(1/T)\!)$ be the completion of $F$ at $\infty$. 
For a non-zero ideal $\fn\lhd A$, by abuse of notation, we denote by the same symbol the 
unique monic polynomial in $A$ generating $\fn$. We will call  a non-zero prime ideal of $A$ simply a \textit{prime} of $A$. 
Given a prime $\fp$ of $A$, we denote by $A_\fp$ (resp. $F_\fp$) the completion of $A$ at $\fp$ (resp. the fraction field of $A_\fp$); 
we also denote $\F_\fp=A/\fp$.  
Given a field $L$, we denote by  $L^\alg$ (resp. $L^\sep$) an algebraic (resp. separable) closure of $L$, and $G_L=\Gal(L^\sep/L)$.

%\textcolor{blue}{ An \textit{$A$-order} in $K$ is a subring $\cO$ of $K$ that (i) contains $A$; 
%(ii) is finitely generated as an $A$-module; (iii) satisfies $F\cO=K$. 
% This definiton is from Gekeler's Helvetichi paper. }

Let $K$ be a field extension of $F$ of degree $r\geq 2$. Let $\cO_K$ be the integral closure of $A$ in $K$. 
An \textit{$A$-order} $\cO$ in $K$ is a subring of $K$ such that 
\begin{itemize}
\item[(i)] $A\subset \cO$;
\item[(ii)] $\cO$ is a finitely generated $A$-module (equiv. $\cO$ is an $A$-subalgebra of $\cO_K$);
% Since $\cO$ is a finitely generated $A$-module, its elements are integral over $A$.
\item[(iii)] $\cO$ contains an $F$-basis of $K$ (equiv. the quotient module $\cO_K/\cO$ has finite cardinality). 
\end{itemize} 
Since $A$ is a PID, $\cO$ is a free $A$-module of rank 
$$\rank_A \cO= \dim_F(\cO\otimes_A F)= \dim_F K=r.
$$

Let $\cO\subset \cO'$ be two $A$-orders in $K$. Since both modules $\cO$ and $\cO'$ have the same rank over $A$, and both contain $1$, 
we have
\begin{equation}\label{eq1}
\cO'/\cO\cong A/\fb_1 \times A/\fb_2 \times \cdots \times A/\fb_{r-1},
\end{equation}
for uniquely determined non-zero ideals $\fb_1, \dots, \fb_{r-1}\lhd A$ such that 
$$
\fb_1\mid \fb_2\mid \cdots \mid \fb_{r-1}. 
$$
We call $$\chi(\cO'/\cO)=\prod_{i=1}^{r-1}\fb_i$$ 
the \textit{index} of $\cO$ in $\cO'$, and $(\fb_1, \dots, \fb_{r-1})$ the \textit{refined index}.  
%cf. \cite[p. 47]{SerreLF}. 
(Note that $\chi(\cO'/\cO)$ is the Fitting ideal of the $A$-module $\cO'/\cO$.) %; our notation follows the notation in Serre's book \cite[p. 16]{SerreLF}.)
The \textit{conductor} of $\cO$ in $\cO'$ is 
$$
\fc=\{c\in K\mid c \cO'\subset \cO\}. 
$$
This is the largest ideal in $\cO'$ which is also an ideal in $\cO$. The conductor is related to the refined index  
by $\fc\cap A=\fb_{r-1}$. % cf. Breuer's Crelle paper p. 41

\subsection{Drinfeld modules} Let $L$ be a field provided with a structure $\gamma: A\to L$ of an 
$A$-algebra. Note that either $\gamma$ is injective or $\gamma$ factors through the quotient map $A\to \F_\fp\hookrightarrow L$  
for some prime $\fp$. 
Let $\tau$ be the Frobenius 
endomorphism of $L$ relative to $\F_q$, that is, the map $\alpha\mapsto \alpha^q$, and let $L\{\tau\}$ 
be the noncommutative ring of polynomials in $\tau$ with coefficients in $L$ 
and the commutation rule $\tau c=c^q\tau$ for any $c\in L$. 

A \textit{Drinfeld $A$-module over $L$ of rank $r\geq 1$} is a ring homomorphism 
\begin{align*}
\phi: A &\To L\{\tau\} \\ 
a &\mapsto \phi_a=\gamma(a)+\sum_{i=1}^{r\cdot \deg(a)} g_i(a) \tau^i 
\end{align*}
uniquely determined through 
$$
\phi_T=\gamma(T)+\sum_{i=1}^{r} g_i(T) \tau^i,\quad g_r(T)\neq 0. 
$$

A \textit{morphism} from the Drinfeld module $\phi$ to the Drinfeld module $\psi$ over $L$ 
is some $u\in L\{\tau\}$ such that $u\phi_a=\psi_a u$ for all $a\in A$ (it suffices to require this 
for $a=T$); $u$ is an \textit{isomorphsim} if $u\in L\{\tau\}^\times=L^\times$.  With this definition, the \textit{endomorphism ring} 
$$
\End_L(\phi)=\{u\in L\{\tau\}\mid u\phi_T=\phi_T u\}
$$
is the centralizer of $\phi(A)$ in $L\{\tau\}$. It is clear that $\End_L(\phi)$ contains in its center 
the subring $\phi(A)\cong A$, hence is an $A$-algebra. 
It can be shown that $\End_L(\phi)$ is a free $A$-module of rank $\leq r^2$, and $\End_L(\phi)\otimes_A \Fi$ is a 
division algebra over $\Fi$; see \cite{Drinfeld}. 

The Drinfeld module $\phi$ endows $L^\alg$ with an $A$-module structure, where $a\in A$ 
acts by $\phi_a$. The \textit{$a$-torsion} $\phi[a]\subset L^\alg$ of $\phi$ is the kernel of $\phi_a$, that is, the set of zeros 
of the polynomial 
\begin{equation}\label{eqPola}
\phi_a(x)=\gamma(a)x+\sum_{i=1}^{r\cdot \deg(a)}g_i(a)x^{q^i}\in L[x]. 
\end{equation}
It is clear that $\phi[a]$ 
has a natural structure of an $A$-module. 
If $a$ is not divisible by $\ker(\gamma)$, then 
$\phi[a]\cong (A/aA)^{\oplus r}$ and $\phi[a]\subset L^\sep$ (since $\phi'_a(x)=\gamma(a)\neq 0$). 
%On the other hand, if $a=\ker(\gamma)$, then $\phi[a]\cong (A/aA)^{\oplus r'}$ for some $0\leq r'<r$. 
For a prime $\fl\lhd A$ 
different from $\ker(\gamma)$, the \textit{$\fl$-adic Tate module of $\phi$},   
$$
T_\fl(\phi)=\underset{\longleftarrow}{\lim}\ \phi[\fl^n]\cong A_\fl^{\oplus r},  
$$
carries a continuous Galois representation 
$$
\rho_{\phi, \fl}: G_L\to \Aut_{A_\fl}(T_\fl(\phi))\cong \GL_r(A_\fl). 
$$

\subsection{Main results} Let $\phi: A\to F\{\tau\}$ be a Drinfeld module of rank $r$ over $F$ defined by 
$$
\phi_T=T+g_1\tau+\cdots+g_r\tau^r. 
$$
(We will always implicitly assume that $\gamma: A\to F$ is the canonical embedding of $A$ into its field of fractions.)
%From now on we assume that $r\geq 2$. 
We say that a prime $\fp\lhd A$ is a \textit{prime of good reduction} for $\phi$ if $\ord_\fp(g_i)\geq 0$ for $1\leq i\leq r-1$, 
and $\ord_\fp(g_r)=0$. In that case, we can consider $g_1, \dots, g_r$ as elements of $A_\fp$; denote by 
$\overline{g}$ the image of $g\in A_\fp$ under the canonical homomorphism $A_\fp\to A_\fp/\fp$. 
The \textit{reduction of $\phi$ at $\fp$} is the 
Drinfeld module $\phi\otimes\F_\fp$ over $\F_\fp$ given by 
$$
(\phi\otimes\F_\fp)_T=\overline{T}+\overline{g_1}\tau+\cdots+\overline{g_r}\tau^r.  
$$
Note that $\phi\otimes\F_\fp$ has rank $r$ since $\overline{g_r}\neq 0$. All but finitely many primes 
of $A$ are primes of good reduction for a given Drinfeld module $\phi$; we denote the set of these primes by $\cP(\phi)$. 

Let $\fp\in \cP(\phi)$ and $d=\deg(\fp)$. Let $\cE_\fp:=\End_{\F_\fp}(\phi\otimes\F_\fp)$. 
It is easy to see that $\pi_\fp:=\tau^d$ is in the center of $\F_\fp\{\tau\}$, hence $\pi_\fp\in \cE_\fp$. 
Using the theory of Drinfeld modules over finite fields it is easy to show that $A[\pi_\fp]\subset \cE_\fp$ 
are $A$-orders in an \textit{imaginary} field extension $K:=F(\pi_\fp)$ of $F$ of degree $r$ (``imaginary" means that 
there is a unique place of $K$ over $\infty$); see 
Proposition \ref{thm1}. Denote by $\cO_{\fp}$ the integral closure of $A$ in $K$. Thus, we have the inclusion of $A$-orders 
$$A[\pi_\fp]\subset \cE_\fp\subset \cO_\fp.$$ 

The endomorphism rings (and algebras) of Drinfeld modules 
over finite fields have been extensively studied; cf. \cite{Drinfeld2}, \cite{Angles}, \cite{GekelerFDM}, \cite{GekelerTAMS}, \cite{JKYu}. 
They play an important role in the theory of Drinfeld modules, as well as their applications to other areas, such as 
the theory of central simple algebras (cf. \cite{GekelerCR}) or the Langlands conjecture over function fields (cf. \cite{Drinfeld2}, \cite{LaumonCDV}). 
In this paper, we are primarily interested in the indices of $\cE_\fp/A[\pi_\fp]$ and $\cO_\fp/\cE_\fp$ as $\fp$ varies. 
We prove the following: 
\begin{thm}\label{thm11} Let $\phi$ be a Drinfeld $A$-module of rank $r\geq 2$ over $F$. 
Let $\fn$ and $\fm$ be arbitrary non-zero elements of $A$. 
\begin{enumerate}
\item The subset of primes $\fp\in \cP(\phi)$  such that $\fn$ divides $\chi(\cE_\fp/A[\pi_\fp])$ has positive density. 
\item If $\End_{F^\alg}(\phi)=A$, then the subset of primes $\fp\in \cP(\phi)$ such that $\fn$ divides $\chi(\cE_\fp/A[\pi_\fp])$ and, simultaneously,  
$\fm$ divides $\chi(\cO_\fp/\cE_\fp)$ has positive density.  
\end{enumerate}
\end{thm}
We prove (1) as  Corollary \ref{cor3.2}, and the proof of (2) is given at the end of Section \ref{ss4.2}. 
%In particular, $\cE_\fp$ can be arbitrarily larger that $A[\pi_\fp]$ or arbitrarily smaller than $\cO_\fp$ for infinitely many $\fp$. 
In fact, we prove stronger results about the refined indices from which this theorem follows. %; see Section \ref{ss4.2}.  
Our proof is modeled on the proof of an analogous result for abelian varieties by Zarhin \cite{Zarhin}. 
%We note that the assumption $\End_{F^\alg}(\phi)=A$ in Theorem \ref{thm11} (2) is necessary. 

Next, let $\fp\in \cP(\phi)$ and $(\fb_{\fp,1}, \dots, \fb_{\fp,r-1})$ be the refined index of $\cE_\fp/A[\pi_\fp]$. 
Let 
$$
P_\fp(X)=X^r+a_{\fp, 1}X^{r-1}+\cdots +a_{\fp, r-1}X+a_{\fp, r} \in A[X]
$$ the minimal polynomial of $\pi_\fp$ over $F$.
We show that $\fb_{\fp,1}$ and $a_{\fp,1}$ produce an interesting reciprocity law (see Section \ref{sRL}). 

\begin{thm}\label{thm12} Let $\phi$ be a Drinfeld $A$-module of rank $r\geq 2$ over $F$.  Let $0\neq \fn \lhd A$. 
Assume the characteristic $p$ of $F$ does not divide $r$ and the prime $\fp\in \cP(\phi)$ does not divide $\fn$. Then 
$\fp$ splits completely in the Galois extension $F(\phi[\fn])$ of $F$ if and only if 
$$
a_{\fp, 1}+r\equiv 0\ (\mod\ \fn)\quad \text{and}\quad \fb_{\fp, 1}\equiv 0\ (\mod\ \fn). 
$$
\end{thm}
Theorem \ref{thm12} for $r=2$ was proved in \cite{CP}. Here we give a somewhat different and simpler proof which 
works for any $r$. The restriction on $r$ being coprime to $p$ 
can be removed if $\fn$ is prime; see Remark \ref{rem3.2}. The primes which split completely in $F(\phi[\fn])$ have been 
 studied before, in particular in papers by Cojocaru and Shulman \cite{CS1}, \cite{CS2}, and Kuo and Liu \cite{KL}. 
 We also note that the argument of the proof 
 of Theorem \ref{thm12} can be adapted to the setting of elliptic curves over $\Q$ to give
a different (simpler) proof of \cite[Cor. 2.2]{DT} which does not rely on Deuring's Lifting Theorem. 

In Section \ref{sA}, we discuss how to compute in practice the endomorphism ring $\cE_\fp$ 
and the indices $\chi(\cE_\fp/A[\pi_\fp])$ and $\chi(\cO_\fp/\cE_\fp)$. 
The calculation of the characteristic polynomial of the Frobenius $P_\fp(X)$ is fairly straightforward, 
and easier than Schoof's algorithm \cite{SchoofEC} for elliptic curves over $\F_p$.  %or abelian varieties , \cite{Pila}.  
We describe an algorithm for calculating $P_\fp(X)$ in polynomial time in $d$ and $r$. Then, assuming the rank of $\phi$ is $2$, 
we describe an algorithm for computing $\cE_\fp$. 
The algorithm actually computes all three rings $A[\pi_\fp]\subset \cE_\fp\subset \cO_\fp$, 
and the corresponding indices. In comparison to the known algorithms for computing 
the endomorphism rings of elliptic curves over finite fields, cf. \cite{DT}, \cite{Centeleghe}, our algorithm is quite 
different and more elementary. The difference stems from the fact that we crucially use the fact that 
$\End_{\F_\fp}(\phi)$ is a subring of the larger ambient space $\F_\fp\{\tau\}$. 
We have implemented our algorithms in \texttt{Magma}, and the examples 
presented in Section \ref{sA} are based on computer calculations. 

\section{Drinfeld modules over $\F_\fp$}\label{ssFDM}
In this section we collect some facts about Drinfeld modules over finite fields that are used throughout the paper. 

Let $\fp\lhd A$ be a prime of degree $d$. Let $\phi: A\to \F_\fp\{\tau\}$ be a Drinfeld $A$-module over $\F_\fp$ of rank $r$ 
determined by 
$$
\phi_T=\gamma(T)+g_1\tau+\cdots+g_{r-1}\tau^{r-1}+g_r\tau^r, \quad g_r\neq 0, 
$$
where $\gamma: A\to \F_\fp=A/\fp$ is the quotient map. 
It is easy to see that $\pi:=\tau^d$ is in the center of $\F_\fp\{\tau\}$, hence $\pi\in \End_{\F_\fp}(\phi)$. 
Denote by $P(X)\in A[X]$ the minimal polynomial of $\pi$ over $\phi(A)$. 

\begin{prop}\label{thm1} Let $K=F(\pi)\cong F[X]/(P(X))$. 
\begin{enumerate}
\item The field extension $K/F$ is imaginary of degree $r$.%, i.e., the place $\infty$ does not split in $K/F$.  
\item $\End_{\F_\fp}(\phi)$ is an $A$-order in $K$. 
\end{enumerate}
\end{prop}
\begin{proof} Let $r_1$ be the degree of $P(X)$. % (equiv. the degree $[K:F]$). 
By Theorem 2.9 in \cite{GekelerFDM}, $r_1$ divides $r$. Let $r_2=r/r_1$. 
Let $P_1(X)=P(X)^{r_2}$. By Lemma 3.3 and Theorem 5.1 (ii) in \cite{GekelerFDM}, we 
have $(P_1(0))=\fp$. Thus, if $c=P(0)$ is the constant term of $P(X)$, then $c^{r_2}$, up to $\F_q^\times$ multiple, 
is equal to the irreducible monic polynomial $\fp$. This implies $r_2=1$, or equivalently $r=r_1$. 
By \cite[(2.3)]{GekelerFDM}, $K/F$ is imaginary. This proves (1). 

By Theorem 2.9 in \cite{GekelerFDM}, $\End_{\F_\fp}(\phi)\otimes_A F$ is a central division algebra 
over $K$ of dimension $r_2^2=1$. Thus, $\End_{\F_\fp}(\phi)\otimes_A F = K$. This proves (2). 
\end{proof}

The previous proposition implies that $P(X)$ has degree $r$. Write 
$$
P(X)=X^r+a_1 X^{r-1}+\cdots+a_r . 
$$
\begin{prop}\label{prop1.3}  For $1\leq i\leq r$, we have $$\deg(a_i)\leq \frac{i\cdot  d}{r}.$$ 
\end{prop}
\begin{proof}
This follows from \cite[Thm. 1 (f)]{JKYu}.
\end{proof}

\begin{prop}\label{thm4} 
Let 
$$
\eps(\phi):=(-1)^r (-1)^{d(r+1)}\Nr_{\F_\fp/\F_q}(g_r)^{-1} \in \F_q^\times. 
$$
Then 
$$
a_r=\eps(\phi)\fp. 
$$
\end{prop}
\begin{proof}
This follows from \cite[p. 268]{HYu}. % and \cite[Thm. 5.1 (ii)]{GekelerFDM}. 
\end{proof}

We can consider $\F_\fp$ as an $A$-module via $\phi$; this module will be denoted ${^\phi}\F_\fp$. Then 
\begin{equation}\label{eqdecomposition}
{^\phi}\F_\fp\cong A/\fd_1\times \cdots \times A/\fd_r,
\end{equation}
for uniquely determined non-zero ideals $\fd_1, \dots, \fd_r\lhd A$ such that 
$\fd_1\mid \fd_2\mid \cdots\mid \fd_r$. There are at most $r$ terms because ${^\phi}\F_\fp$ 
is a finite $A$-module so for some $\fd\lhd A$ we have ${^\phi}\F_\fp\subset \phi[\fd]\cong (A/\fd)^r$. 
Denote $\chi(\phi):=\fd_1\cdots \fd_r$. It is clear that $\deg\chi(\phi)=\deg\fp$. %(but generally $\chi(\phi)\neq \fp$). 
It is also easy to see the following:
\begin{lem}\label{lem1.5} $\fd_1\in A$ is the monic polynomial of largest degree such that  
	$\fp$ does not divide $\fd_1$ and all the roots of $\phi_{\fd_1}(x)$ are in $\F_\fp$. 
 \end{lem}
 %\begin{rem} Note that $\chi(\phi)=\fp$ if and only if $\phi[\fp](\F_\fp)\cong A/\fp$. For example, if $q=2$, then $\phi_T=\tau+\tau^2$ 
 %over $\F_T$ has this property. 
 %\end{rem}

\begin{prop}\label{prop1.7} We have 
$$
\chi(\phi)=P(1)A=(1+a_1+a_2+\cdots + a_r)A.
$$
\end{prop}
\begin{proof}
See \cite[Thm. 5.1 (i)]{GekelerFDM}. 
\end{proof}

\begin{comment}
\begin{thm}\hfill 
\begin{enumerate}
\item The principal ideal $(P(1))$ of $A$ equals $\chi(\phi)$. 
\item Let $x$ be a zero of $P$. Under any embedding of $F(x)$ into $\C_\infty$ we have $|x|_\infty=|\fp|^{1/r}$. 
\end{enumerate}
\end{thm}
\begin{proof}
(1) is \cite[Thm. 5.1 (i)]{GekelerFDM} and (2)  follows from \cite[Thm. 1 (f)]{JKYu}. 
\end{proof}
\end{comment}

\begin{prop}\label{prop1.8} Let $\Frob_\fp\in G_{\F_\fp}$ be the Frobenius automorphism $\alpha\mapsto \alpha^{q^d}$. 
Let $\fl\lhd A$ be a prime different from $\fp$. 
\begin{enumerate}
\item The characteristic polynomial of $\rho_{\phi,\fl}(\Frob_\fp)$ is $P(X)$; in particular, the characteristic polynomial 
of $\rho_{\phi,\fl}(\Frob_\fp)$ has coefficients in $A$ and does not depend on $\fl$. 
\item The natural map 
$$
\End_{\F_\fp}(\phi)\otimes_A A_\fl\to \End_{A_\fl[G_{\F_\fp}]}(T_\fl(\phi))
$$
is an isomorphism. 
\end{enumerate}
\end{prop}
\begin{proof}
See \cite[$\S$3]{GekelerFDM} or \cite[Thm. 2]{JKYu}. 
\end{proof}

\section{Reciprocity law: Proof of Theorem \ref{thm12}}\label{sRL} 
Let $\phi$ be a Drinfeld $A$-module over $F$ of rank $r\geq 2$. 
For $\fn\lhd A$, let $F(\phi[\fn])$ be the splitting field of the polynomial $\phi_\fn(x)$ in \eqref{eqPola}. 
This is a Galois extension whose Galois group is naturally a subgroup of $\GL_r(A/\fn)$. A prime $\fp\lhd A$
is unramified in $F(\phi[\fn])/F$ if $\fp\in \cP(\phi)$ and $\fp$ 
does not divide $\fn$; cf. \cite{Takahashi}. Theorem \ref{thm12} 
describes those primes which split completely in $F(\phi[\fn])$ in terms of congruences modulo $\fn$; such theorems 
are usually called ``reciprocity laws''. Recall that the set of all but finitely many primes which split completely in a 
given Galois extension uniquely determines that extension. In the following proof we keep the notion introduced right before 
Theorem \ref{thm12}. 

\begin{comment}
For $\fp\in \cP(\phi)$, let $\pi_\fp$ be the Frobenius endomorphism of $\phi\otimes\F_\fp$ and 
$$
P_\fp(X)=X^r+a_{\fp, 1}X^{r-1}+\cdots +a_{\fp, r-1}X+a_{\fp, r}
$$
be the minimal polynomial of $\pi_\fp$ over $F$; cf. Section \ref{ssFDM}. 
Denote $\cE_\fp=\End_{\F_\fp}(\phi\otimes\F_\fp)$. 
Then we have the inclusion of orders $A[\pi_\fp]\subseteq \cE_\fp$ in an imaginary field extension of $F$ of degree $r$. Let 
$(\fb_{\fp,1}, \dots, \fb_{\fp,r-1})$ be the refined index of $A[\pi_\fp]$ in $\cE_\fp$. 

\begin{thm}\label{thm2.1}
Assume the characteristic $p$ of $F$ does not divide $r$, and $\fp\in \cP(\phi)$ does not divide $\fn$. Then 
$\fp$ splits completely in $F(\phi[\fn])$ if and only if 
$$
a_{\fp, 1}+r\equiv 0\ (\mod\ \fn)\quad \text{and}\quad \fb_{\fp, 1}\equiv 0\ (\mod\ \fn). 
$$
\end{thm}
\end{comment}
\begin{proof}[Proof of Theorem \ref{thm12}] 
	To simplify the notation, denote $\overline{\phi}=\phi\otimes\F_\fp$. 
	%The action of $\Gal(F(\phi[\fn])/F)$ on $\phi[\fn]$ gives an injection 
%$$\rho: \Gal(F(\phi[\fn])/F)\to \Aut(\phi[\fn])\cong \GL_r(A/\fn).$$ 
%Let $\fp$ be a prime not dividing $\fn g_r$. Then $\fp$ is unramified in $F(\phi[\fn])$. 
The prime $\fp$ is unramified in $F(\phi[\fn])$.  
Reducing $\phi_\fn(x)$ modulo $\fp$ we get a canonical 
isomorphism $\phi[\fn]\cong \overline{\phi}[\fn]$ of $A$-modules. The prime $\fp$ splits 
completely in $F(\phi[\fn])$ if and only if $\Frob_\fp\in G_{\F_\fp}$ acts as the identity on $\overline{\phi}[\fn]$. 
On the other hand, the action of $\Frob_\fp$ on $\overline{\phi}[\fn]$ agrees with the action of 
$\pi_\fp$ as an endomorphism of $\overline{\phi}$. Thus, we need to show that $\pi_\fp$ acts as the 
identity on $\overline{\phi}[\fn]$ if and only if the congruences of the theorem hold. 

%Let $\cE:=\End_{\F_\fp}(\phi)$. By assumption, $\fp\nmid \fn$, as otherwise $\fp$ ramifies in $F_\fn$. 
First, we prove that $w\in \cE_\fp$ acts as $0$ on $\overline{\phi}[\fn]$ 
if and only if $w\in \fn\cE_\fp$. If $w\in \fn\cE_\fp$ then $w=v\overline{\phi}_\fn$ for some $v\in \cE_\fp$, so 
it obviously acts as $0$ on $\overline{\phi}[\fn]$. Conversely, suppose $w$ acts as $0$ on $\overline{\phi}[\fn]$. By the 
division algorithm in $\F_\fp\{\tau\}$, we can write $w=v\overline{\phi}_\fn+u$ for some $v, u\in \F_\fp\{\tau\}$  with $u=0$ or $\deg_\tau(u)<\deg_\tau(\overline{\phi}_\fn)$. 
Since $w$ and $\overline{\phi}_\fn$ act as $0$ on $\overline{\phi}[\fn]$, so does $u$. On the other hand, the polynomial 
$\overline{\phi}_\fn(x)$ is separable. This implies $\deg_\tau(u)\geq \deg_\tau(\overline{\phi}_\fn)$ or $u=0$. Thus, $u=0$ and $w=v\overline{\phi}_\fn$. 
We need to show that $v\in \cE_\fp$, i.e., $v$ commutes with $\overline{\phi}(A)$. 
Now $w\overline{\phi}_b=v\overline{\phi}_\fn\overline{\phi}_b=v\overline{\phi}_b\overline{\phi}_\fn$ and 
$w\overline{\phi}_b=\overline{\phi}_bw=\overline{\phi}_b v \overline{\phi}_\fn$. Thus, 
$(v\overline{\phi}_b-\overline{\phi}_bv)\overline{\phi}_\fn=0$. Since $\F_\fp\{\tau\}$ has no zero-divisors 
and $\overline{\phi}_\fn\neq 0$, we must have $v\overline{\phi}_b=\overline{\phi}_bv$ for all $b\in A$, so $v\in \cE_\fp$. 

Suppose $\pi_\fp$ acts as a scalar on $\overline{\phi}[\fn]$. This means that there is $c\in A$ such that $\pi_\fp-c$ annihilates $\overline{\phi}[\fn]$. 
By the previous paragraph, this is equivalent to $\pi_\fp-c$ being in $\fn\cE_\fp$; this is equivalent to $A[\pi_\fp]\subset A+\fn\cE_\fp$. 
We can choose an $A$-basis $1, e_1, \dots, e_{r-1}$ of $\cE_\fp$ such that 
$A[\pi_\fp]=A+\fb_{\fp, 1}e_1+\cdots+\fb_{\fp, r-1}e_{r-1}$.  
%with $b_1\mid b_2\mid\cdots \fb_{r-1}$ (elementary divisor theorem). 
Since $\fb_{\fp, 1}$ divides $\fb_{\fp, 2}, \dots, \fb_{\fp, r-1}$, the inclusion $A[\pi_\fp]\subset A+\fn\cE_\fp$ is equivalent to 
$\fn\mid \fb_{\fp, 1}$. Thus, $\pi_\fp$ acts as a scalar on $\overline{\phi}[\fn]$ if and only if $\fn$ divides $\fb_{\fp, 1}$. 

Now note that $\pi_\fp$ acts as the identity on $\overline{\phi}[\fn]$ if and only if $\pi_\fp$ acts as a scalar $c$ on $\overline{\phi}[\fn]$ 
and $c\equiv 1\ (\mod\ \fn)$. If $\pi_\fp$ acts as a scalar then the characteristic polynomial of $\pi_\fp$ satisfies 
$P_\fp(X)\equiv (X-c)^r\ (\mod\ \fn)$. (This congruence is not sufficient for $\pi_\fp$ acting as a scalar on $\overline{\phi}[\fn]$, if 
the action is not semi-simple.) We have  
$$
(X-c)^r=  X^r-rcX^{r-1}+\cdots. 
$$
Since $p$ does not divide $r$ by assumption, we see that $c\equiv 1\ (\mod\ \fn)$ if and only if $rc\equiv r\ (\mod\ \fn)$. Hence 
$c\equiv 1\ (\mod\ \fn)$ if and only if $a_{\fp, 1}\equiv -rc\equiv -r\ (\mod\ \fn)$. 
\end{proof}

\begin{rem}\label{rem3.2} If $\fn$ itself is prime, then in Theorem \ref{thm12}
one can dispose with the assumption that $p\nmid r$ as follows: 
Decompose $r=p^sr'$, $s\geq 0$ with $p\nmid r'$. Then  
$\fp\nmid \fn$ splits completely in $F(\phi[\fn])$ if and only if 
$$
a_{\fp, p^s}+r'\equiv 0\ (\mod\ \fn)\quad \text{and}\quad \fb_{\fp, 1}\equiv 0\ (\mod\ \fn). 
$$
The proof is essentially the same except at the end we have 
$$
(X-c)^r=(X^{p^s}-c^{p^s})^{r'}=  X^r-r'c^{p^s}X^{p^s(r'-1)}+\cdots. 
$$
If $\fn$ is prime, then the $p$th power map is an automorphism of $(A/\fn)^\times$. Thus, 
$c\equiv 1\ (\mod\ \fn)$ if and only if $c^{p^s}\equiv 1\ (\mod\ \fn)$; thus, 
$c\equiv 1\ (\mod\ \fn)$ if and only if $a_{\fp, p^s}\equiv -r'c^{p^s}\equiv -r'\ (\mod\ \fn)$. 
\end{rem}

\begin{cor}\label{cor3.2} Let $\phi$ be a Drinfeld $A$-module over $F$ of rank $r\geq 2$. 
For a given $\fn\in A$, the set of primes $\fp\in \cP(\phi)$ such that $\fn$ divides $\fb_{\fp,1}$ 
has positive density. In particular, for any $\fn\in A$ there are infinitely many $\fp\in \cP(\phi)$ 
such that $\fn$ divides the index $\chi(\cE_\fp/A[\pi_\fp])$. 
\end{cor}
\begin{proof} From the proof of of Theorem \ref{thm12} we see that for any prime $\fp$
for which  $\pi_\fp$ acts as a scalar on $(\phi\otimes\F_\fp)[\fn]$ we have $\fn\mid \fb_{\fp,1}$ (this part does not 
use the assumption that $r$ is coprime to $p$). The set of primes that split completely in $F(\phi[\fn])$ have this 
property and positive density $1/[F(\phi[\fn]):F]$ by Chebotarev.   
\end{proof}

Let $\fp\in \cP(\phi)$. As in Section \ref{ssFDM}, we can consider $\F_\fp$ as an $A$-module via $\phi\otimes \F_\fp$. Let 
$$
{^{\phi \otimes\F_\fp}}\F_\fp\cong A/\fd_{\fp, 1}\times \cdots \times A/\fd_{\fp, r} 
$$
be the isomorphism of \eqref{eqdecomposition}. (Keep in mind that $\phi$ in Section \ref{ssFDM} is over $\F_\fp$, whereas in this section 
$\phi$ is over $F$.)

\begin{cor}\label{cor3.3} With notation and assumptions of Theorem \ref{thm12}, we have 
$$
\fd_{\fp, 1}=\mathrm{gcd}(\fb_{\fp,1}, a_{\fp, 1}+r). 
$$
\end{cor}
\begin{proof}
From the proof of Theorem \ref{thm12} we see that 
$(\phi \otimes\F_\fp)[\fn]\subset\F_\fp$ if and only if $\fp$ splits completely in $F(\phi[\fn])$.  
The claim then follows from Lemma \ref{lem1.5} and Theorem \ref{thm12}. 
\end{proof}

The previous corollary for $r=2$ already appears in \cite{CP}. In the $r=2$ case, from Corollary \ref{cor3.3}, Proposition \ref{thm4} 
and Proposition \ref{prop1.7} we also get that 
$$
\frac{1+a_{\fp, 1}+\eps(\phi\otimes\F_\fp)\fp}{\mathrm{gcd}(\fb_{\fp,1}, a_{\fp, 1}+r)} \in A, 
$$
and this polynomial generates $\fd_{\fp, 2}$.

%---------------------------------------------------

\section{Large Indices: Proof of Theorem \ref{thm11}} 

\subsection{Preliminaries} Fix a prime $\fl\lhd A$ and denote $K:=F_\fl$. 
Let $V$ be a vector space of dimension $r$ over $K$. 
For $u\in \End_K(V)$, denote by $\Delta(u)$ the discriminant of the characteristic polynomial 
of $u$. Note that the characteristic polynomial of $u$ has no multiple roots over $K^\alg$ 
if and only if $\Delta(u)\neq 0$. 

Denote by $Z(u)$ the centralizer of $u$ in $\End_K(V)$. Assume $\Delta(u)\neq 0$. 
Obviously, $K[u]\subseteq Z(u)$, but since $K[u]$ is a maximal torus in $\End_K(V)$ and $Z(u)=Z(K[u])$, we in fact have 
$K[u]= Z(u)$. Denote 
$$
Z(u)^\circ=\{w\in Z(u)\mid \Delta(w)\neq 0\}. 
$$ 
Since the complement of $Z(u)^\circ$ in $Z(u)$ is the locus of vanishing of $\Delta$, $Z(u)^\circ$ is open and 
everywhere dense in $Z(u)$ with respect to the $\fl$-adic topology. It is clear that for any $w\in Z(u)^\circ$ 
we have $Z(w)=Z(u)$. %, since we have $K[w]\subset Z(w)\subseteq Z(u)=K[u]$ and $\dim K[w]=\dim K[u]=r$. 

\begin{lem}\label{lem3.1}
The map 
$$
\Psi_u: \Aut_K(V)\times Z(u)^\circ \to \Aut_K(V), \quad (g, w)\mapsto gwg^{-1}
$$
is an open map with respect to the $\fl$-adic topology, i.e., the image under $\Psi_u$ of any open subset of 
$\Aut_K(V)\times Z(u)^\circ$ is open in $\Aut_K(V)$. 
\end{lem}
\begin{proof}
It is enough to prove that the induced map on tangent spaces of the corresponding $\fl$-adic manifolds is surjective; cf. 
\cite[Cor. 4.5]{Schneider}. 
Since the map on tangent spaces is a homomorhism of vector spaces over $K$, the property of this map being surjective 
is invariant under base change. Thus, after possibly extending the base field $K$, we can assume that $u$ 
has all its eigenvalues in $K$, so $Z(u)$ is a split torus. Then, after fixing an appropriate basis of $V$, we 
identify $\Aut_K(V)$ with $\GL_r(K)$ and $Z(u)$ with diagonal matrices in $M_r(K)$. 
We show that for any $(g, t)\in \Aut_K(V)\times Z(u)^\circ$, the 
derivative 
$$
d\Psi_u\big|_{(g,t)}: M_r(K)\times Z(u) \to M_r(K)
$$
is surjective. 

A small calculation shows that $d\Psi_u\big|_{(g,t)}(M,N)=gNg^{-1}+Mtg^{-1}-gtg^{-1}Mg^{-1}$. 
Substituting $M'=g^{-1}M$, we can write $d\Psi_u\big|_{(g,t)}(M,N)=g(N+[M', t])g^{-1}$. 
Since we assumed $u$ to be diagonal, $t$ is also a diagonal matrix with \textit{distinct} entries. 
It would be enough to show that for any $X\in M_r(K)$, we can find $M'$ and $N$, such that 
$(N+[M', t])=g^{-1}Xg=:X'$. 
Any $X'\in M_r(K)$ can be written as $X'=X_1+X_2$, where $X_1$ is a diagonal matrix  
and $X_2$ has zeros on the diagonal. We can solve for $M'=(m_{ij})$ such that $[M', t]=(m_{ij}(t_i-t_j))=X_2$ 
since $t_i\neq t_j$ for $i\neq j$. Finally, if we take $M=gM'$ and $N=X_1$ then we get 
$d\Psi_u\big|_{(g,t)}(M,N)=gX'g^{-1}=X$.   
\end{proof}

Let $\La$ be an $A_\fl$-lattice in $V$ of maximal rank $r$. 
Consider the intersection 
$$
\cZ(u):=Z(u)\bigcap \End_{A_\fl}(\La)\subset \End_{A_\fl}(\La)\subset \End_K(V). 
$$
It is clear that $\cZ(u)$ coincides with the centralizer of $u$ in $\End_{A_\fl}(\La)$ and is an ${A_\fl}$-order in $Z(u)$. 

%Clearly $\cZ(u)$ coincides with the centralizer of $u$ in $\End_{A_\fl}(\La)$. It is also clear that 
%$\cZ(u)$ is a ${A_\fl}$-subalgebra (order) in $Z(u)$ and the natural map 
%$$\cZ(u)\otimes_{A_\fl} K\to Z(u), \quad u\otimes c\mapsto cu$$ is an isomorphism of $K$-algebras.

\begin{prop}\label{keyProp}
Let $G$ be an open compact subgroup of $\Aut_{A_\fl}(\La)$. The set 
$$
U=\{\gamma\in G\mid \cZ(\gamma)\cong \cZ(u)\}
$$
is open in $G$. 
\end{prop}
\begin{proof}
Since $G$ is an open compact subgroup of $\Aut_K(V)$, the intersection 
$$
Z(u)^\circ_G:=G \bigcap Z(u)^\circ
$$
is an open subset in $Z(u)^\circ$. (This subset is non-empty since its closure contains the identity). 
Therefore, by Lemma \ref{lem3.1}, $G':=\Psi_u(G\times Z(u)^\circ_G)$ is an open subset of $\Aut_K(V)$. 
On the other hand, obviously, $G'\subset G$. 

Note that for any $g\in G$ and $w\in Z(u)^\circ$, we have 
$$Z(gwg^{-1})=gZ(w)g^{-1}=gZ(u)g^{-1}.$$
Since $G$ is a subgroup of $\Aut_{A_\fl}(\La)$, this implies
$$
\cZ(gwg^{-1})=g\cZ(w)g^{-1}=g\cZ(u)g^{-1}. 
$$
In particular, $\cZ(\gamma)\cong \cZ(u)$ for $\gamma=gwg^{-1}$. Since every element in $G'$ 
is of this form, we conclude that $\cZ(\gamma)\cong \cZ(u)$ for all $\gamma\in G'$. 
As $G'$ is open in $G$, this finishes the proof. 
\end{proof}

\subsection{Main theorem}\label{ss4.2} Let 
$\phi$ be a Drinfeld $A$-module of rank $r\geq 2$ over $F$. 
Let $\fp\in \cP(\phi)$ and $\cE_\fp:=\End_{\F_\fp}(\phi\otimes\F_\fp)$. 
Let $\fl\lhd A$ be a prime different from $\fp$. 
By \cite{Takahashi}, the Tate module $T_\fl(\phi)$ is unramified at $\fp$, i.e., for any place $\bar{\fp}$ in $F^\sep$ 
extending $\fp$, the inertia group of $\bar{\fp}$ acts trivially on $T_\fl(\phi)$. There is a canonical isomorphism 
$T_\fl(\phi)\cong T_\fl(\phi\otimes\F_\fp)$ which is compatible with the action of 
a Frobenius element $\sigma_\fp$ in the decomposition group of $\bar{\fp}$ on $T_\fl(\phi)$ and the action of $\Frob_\fp\in G_{\F_\fp}$ 
on $T_\fl(\phi\otimes\F_\fp)$; cf. \cite[p. 479]{Takahashi}. Hence using Proposition \ref{prop1.8}, we get 
\begin{align*}
\cE_\fp\otimes_A A_\fl &\cong \End_{A_\fl[G_{\F_\fp}]}(T_\fl(\phi\otimes\F_\fp)) \\ 
& \cong \text{Centralizer of $\Frob_\fp$ in $\End_{A_\fl}(T_\fl(\phi\otimes\F_\fp))$}\\
& \cong \text{Centralizer of $\sigma_\fp$ in $\End_{A_\fl}(T_\fl(\phi))$}.
\end{align*}

Now let $C_\fl$ be a fixed commutative semi-simple $F_\fl$-algebra of dimension $r$ and let $R_\fl$ be an $A_\fl$-order 
in $C_\fl$. Fix an isomorphism of free $A_\fl$-modules 
$R_\fl\cong T_\fl(\phi)$ and extend it linearly to an isomorphism 
$R_\fl\otimes_{A_\fl} F_\fl=C_\fl\cong V_\fl(\phi):=T_\fl(\phi)\otimes_{A_\fl} F_\fl$. 
Let 
$$
C_\fl\xrightarrow{\iota} \End_{F_\fl}(C_\fl)\cong \End_{F_\fl}(V_\fl(\phi)) 
$$
be the embedding given by multiplication, i.e., $\iota(\alpha)(x)=\alpha x$. We identify $C_\fl$ with its image in 
$\End_{F_\fl}(V_\fl(\phi))$. Since $C_\fl$ is a maximal torus, it coincides with its own centralizer $Z(C_\fl)$
in $\End_{F_\fl}(V_\fl(\phi))$. 
\begin{lem}\label{lem3.3} We have:
\begin{itemize}
\item[(i)]
$R_\fl=\{u\in C_\fl\mid u(T_\fl(\phi))\subset T_\fl(\phi)\}$.
\item[(ii)]
$R_\fl=\{u\in \End_{A_\fl}(T_\fl(\phi))\mid u\in Z(C_\fl)\}$.
\end{itemize}
\end{lem}
\begin{proof}
(i) We clearly have the inclusion $ R_\fl\subset\{u\in C_\fl\mid u(T_\fl(\phi))\subset T_\fl(\phi)\}$. 
For the reverse inclusion, assume 
$u\in C_\fl$ is such that $u(T_\fl(\phi))\subset T_\fl(\phi)$. Then $u\cdot 1=u\in R_\fl$. 
(ii) This follows from (i) since $Z(C_\fl)=C_\fl$. 
\end{proof}

A simple finite-dimensional commutative algebra over $F_\fl$ is just a field extension of $F_\fl$. Thus
$C_\fl=\prod_{i=1}^h K_i$
where each $K_i$ is a finite algebraic field extension of $F_\fl$. We will assume from now on that each $K_i/F_\fl$ 
is a \textit{separable} extension. Then we have the Primitive Element Theorem, so can use Lemma 2.3 in \cite{Zarhin} to 
prove that (use $A$ instead of $\Z$ in the proof)
\begin{lem}
There exists an invertible element $u_0$ of $C_\fl$ such that $C_\fl=F_\fl[u_0]$. 
\end{lem}

Note that the centralizer $Z(u_0)$ of $u_0$ in $\End_{F_\fl}(C_\fl)$ is $Z(C_\fl)=C_\fl$. Hence by Lemma \ref{lem3.3} the 
centralizer $\cZ(u_0)$ of $u_0$ in $\End_{A_\fl}(T_\fl(\phi))$ is $R_\fl$. 
Let $G$ be an open subgroup of $\Aut_{A_\fl}(T_\fl(\phi))$. 
By Proposition \ref{keyProp}, the set 
$$
U=\{u\in G\mid \cZ(u)\cong \cZ(u_0)\}
$$
is open in $G$. 

We will need the following result of Pink \cite{Pink}:
\begin{thm}\label{thmPink} If $\End_{F^\alg}(\phi)=A$, then for any 
finite set $S$ of primes of $A$ the image of the homomorphism 
$$
G_F\to \prod_{\fl\in S} \Aut_{A_\fl}(T_\fl(\phi))
$$
is open. 
\end{thm} 

Assume $\End_{F^\alg}(\phi)=A$. Let $S=\{\fl_1, \dots, \fl_m\}$ be a set of distinct primes. Choose a commutative separable 
semi-simple $F_\fl$-algebra $C_{\fl}$ of dimension $r$ for each $\fl\in S$. Choose an $A_{\fl}$-order 
$R_{\fl}\subset C_{\fl}$ for each $\fl\in S$. Let 
$$
\rho: G_F \to \prod_{\fl\in S} \Aut_{A_{\fl}}(T_{\fl}(\phi))
$$
be the Galois representation arising from the action on $T_{\fl}(\phi)$. 
Let $\prod_{\fl\in S} G_\fl$ be the image of $\rho$. By Theorem \ref{thmPink}, 
$G_\fl$ is an open subgroup of $\Aut_{A_{\fl}}(T_{\fl}(\phi))$. Let $U_\fl\subset G_\fl$ 
be the open subset provided by Proposition \ref{keyProp}. Since by Chebotarev's theorem the Frobenius elements 
are dense in $G_F$, they are also dense in $\prod_{\fl\in S} U_\fl$. In particular, there are infinitely many $\fp$ 
such that the conjugacy class of $\sigma_\fp$ lies in $\prod_{\fl\in S} U_\fl$. Even stronger, the 
set of such primes has positive density by Corollary 2 (b) on page I-8 of \cite{SerreAR}.
Thus, we have proved the following:  

\begin{thm}\label{thm3.5} Assume $\End_{F^\alg}(\phi)=A$. 
The set of primes $\fp\in \cP(\phi)$ such that $\cE_\fp\otimes_A A_{\fl}\cong R_{\fl}$ for all $\fl\in S$ has positive density. 
\end{thm}

Finally, we are ready to prove Theorem \ref{thm11}. 
\begin{proof}[Proof of Theorem \ref{thm11}] Part (1) of the theorem was already proved as Corollary \ref{cor3.2}. 
	Moreover, it easy to see from the proof of Corollary \ref{cor3.2} that to prove part (2) 
	it is enough to show  that for a subset of primes $\fp\in \cP(\phi)$ of positive density 
	$\sigma_\fp$ acts trivially on $\phi[\fn]$, simultaneously with the condition of Theorem \ref{thm3.5}. 
	
Let $\fn=\fq_1^{s_1}\cdots \fq_d^{s_d}$ be the prime decomposition of a given element $\fn\in A$. 
Let $S'=\{\fq_1, \dots, \fq_d\}$. Let  $\prod_{\fl\in S\cup S'} G_\fl$ be the image of $G_F \to \prod_{\fl\in S\cup S'} \Aut_{A_{\fl}}(T_{\fl}(\phi))$
For $\fq\in S'$, let $G_\fq'$ be 
the intersection of $G_\fq$ with the principal congruence 
subgroup of $\GL_r(A_{\fq})$ of level $\fq^{s}$ consisting of matrices which are congruent to $1$ modulo $\fq^{s}$. 
Note that $G_\fq'$ is still open in $\Aut_{A_\fq}(T_{\fq}(\phi))$. For $\fl\in S\setminus S'$, let $G_\fl'=G_\fl$. 
Now to achieve our goal we can simply apply the argument in the proof of Theorem \ref{thm3.5} 
 to $\prod_{\fl\in S\cup S'}G_\fl'$. %; cf. Section \ref{sRL}.  
\end{proof}

%Finally, we observe that Theorem \ref{thm11} follows from this argument. 

%\begin{cor}\label{cor3.6}
%Assume $\End_{F^\alg}(\phi)=A$. For a prime of good reduction of $\phi$ denote by $\cO_{\fp}$ 
%the integral closure of $A$ in $\cE_\fp\otimes_A F$.  
%Let $\fm$ and $\fn$ be fixed elements of $A$. There are infinitely many $\fp$ 
%such that $\fn\mid \chi(\cE_\fp/A[\pi_\fp])$ and $\fm\mid \chi(\cO_\fp/\cE_\fp)$.  
%\end{cor}

\section{Algorithms}\label{sA}

In this section we describe algorithms for computing some of the invariants of Drinfeld modules over finite fields 
discussed in Section \ref{ssFDM}. We have implemented these algorithms in \texttt{Magma}. The examples 
presented in this section are based on computer calculations. 

Throughout this section 
$\phi: A\to \F_\fp\{\tau\}$ is a Drinfeld module of rank $r$, $\fp\lhd A$ is a prime of degree 
$d$, and $\gamma: A\to \F_\fp= A/\fp$ is the reduction modulo $\fp$. 

\subsection{Characteristic polynomial of the Frobenius}
%The calculation of the characteristic polynomial of the Frobenius of $\phi$ is fairly straightforward, 
%and easier than the corresponding algorithms for elliptic curves or abelian varieties over $\F_p$; cf. \cite{SchoofEC}, \cite{Pila}. 
Let
$$
P(X)=X^r+a_1X^{r-1}+\cdots+a_r 
$$
be the characteristic polynomial of $\Frob_\fp$ acting on $T_\fl(\phi)$; cf. Proposition \ref{prop1.8}.  
From Proposition \ref{prop1.3} we know that $a_1, \dots, a_r\in A$ and $\deg(a_i)\leq i\cdot\frac{d}{r}$. 
In particular, $a_1, \dots, a_{r-1}$ 
are uniquely determined by their residues modulo $\fp$. We also know from Proposition \ref{thm4} that 
$
a_r %= (-1)^r(-1)^{d\cdot (r+1)}\Nr_{\F_\fp/\F_q}(\Delta)^{-1}\fp
=\eps(\phi)\fp$. 
Now since $P(X)$ is also the minimal polynomial of $\pi=\tau^d$, we have 
$$
\tau^{dr}+\phi_{a_1}\tau^{d(r-1)}+\cdots +\phi_{a_{r-1}}\tau^d+\phi_{a_r}=0. 
$$ 
Denote 
$$
f_i=\phi_{a_{i}}\tau^{d(r-i)}+\phi_{a_{i+1}}\tau^{d(r-i-1)}+\cdots+\phi_{a_r}\in \F_\fp\{\tau\} 
$$
and 
$$
f_i^{\dag}=\tau^{dr}+\phi_{a_1}\tau^{d(r-1)}+\cdots +\phi_{a_{i-1}}\tau^{d(r-i+1)}.
$$
Note that $\deg_{\tau}\phi_{a_j}\tau^{d(r-j)}\geq d(r-j)$, so the coefficient of $\tau^{d(r-i+1)}$ in $f_{i}^{\dag}$ 
is the constant term of  $\phi_{a_{i-1}}$, i.e., $\gamma(a_{i-1})$. Therefore, 
$$
\gamma(a_{i-1}) = - \text{Coefficient of $\tau^{d(r-i+1)}$ in $f_i$}. 
$$
Since we know $f_r$ explicitly, we can compute all $a_i$ recursively, where we use $a_r, a_{r-1}, \dots, a_{r-i}$ 
to calculate $a_{r-i-1}$. 

Computing $\phi_a$ takes approximately $r\deg(a)^2$ operations (computing $\phi_{T^n}$ recursively via 
$\phi_T\phi_{T^{n-1}}$ takes $\approx nr$ operations, 
so computing $\phi_a$ takes $r(\deg(a)+(\deg(a)-1)+\cdots+1\approx r\deg(a)^2$ operations). 
We conclude that the amount of work involved in the calculation of $P(X)$ is $O(r^2d^2)$, so this is a 
``polynomial time" algorithm; cf. \cite{SchoofEC}.  

\begin{example}
Let $q=3$, $\fp=T^7 - T^2 + 1$, and $\phi_T=T+(T^2+1)\tau+T\tau^2+\tau^3$. Then 
$$
P(X)=X^3+(-T + 1)X^2+(T^3 + T -1)X-\fp. 
$$
\end{example}

\begin{rem} For rank $r=2$ there is a different recursive procedure for computing $P(X)$ 
based on the properties of Eisenstein series; see \cite[Prop. 3.7]{GekelerTAMS}. In practice, Gekeler's  
algorithm seems to have the same efficiency as what was presented above (the amount of computer time 
it took to execute both in \texttt{Magma} were almost identical in our tests). 
\end{rem}

%\begin{example}
%Let $q=3$, $\fp=T^7 - T^2 + 1$, and $\phi_T=T+T\tau+\tau^3$. Then 
%$$P(X)=-\fp+(T^3 - T^2 - T + 1)X+(T + 1)X^2+X^3. $$
%\end{example}

\subsection{Exponent of ${^\phi}\F_\fp$} 
Let 
$$
{^\phi}\F_\fp\cong A/\fd_1\times \cdots \times A/\fd_r,
$$
be the isomorphism of \eqref{eqdecomposition}. We call $\fd_r$ the \textit{exponent} of ${^\phi}\F_\fp$, 
since the fact that $\fd_1\mid\cdots\mid \fd_r$ implies that $\fd_r$ 
is the smallest degree element of $A$ such that $\phi_{\fd_r}$ annihilates $\F_\fp$. The exponent was 
studied in prior papers by Cojocaru and Shulman \cite{CS1}, \cite{CS2}. 

%Now note that $P(1)$ is a polynomial of degree $d$ with leading coefficient equal to $\eps(\phi)$, thus 
%$\chi(\phi)=\eps(\phi)^{-1}P(1)$. In principle, this allows to calculate the exponential of ${^\phi}\F_\fp$ by factoring 
%$\chi(\phi)$, but there is an easier direct method. 

%Treat $\F_\fp$ as a quotient module $A/\fp$ (the multiplicative structure is irrelevant in this problem). 
Denote $\theta=\gamma(T)$. Then $\F_\fp=A/\fp$ is an $\F_q$-vector space 
with basis $1, \theta, \theta^2, \dots, \theta^{d-1}$. 
Put 
$$
\fd=x_0+x_1T+\cdots +x_{d-1}T^{d-1}+T^d
$$
for a hypothetical annihilator of ${^\phi}\F_\fp$, i.e., $\phi_{\fd}$ acts as zero on $\F_\fp$. Let $k\in \{0, \dots, d-1\}$. Compute 
$$
\phi_{T^i}(\theta^k)=\alpha_{i,1}+\alpha_{i,2}\theta+\cdots+\alpha_{i,d}\theta^{d-1} 
$$
for $i=1,2,\dots, d$, which can be easily done by repeated application of $\phi_T$. Let $M_k$ 
be the $d\times d$ matrix whose first row consists of zeros except at position $k+1$ where it is $1$, and the 
$(i+1)$-th row is $[\alpha_{i,1}, \alpha_{i,2},\dots, \alpha_{i,d}]$ for $1\leq i\leq d-1$. Let 
$N_k=-[\alpha_{d, 1}, \dots, \alpha_{d, d}]^t$. Then $\phi_\fd$ acting as $0$ on $A/\fp$ 
is equivalent to $\phi_{\fd}(\theta^k)=0$ for all $k=0, \dots, d-1$, which itself is equivalent to  
$$
[x_0, \dots, x_{d-1}]M_k=N_k \quad \text{for all}\quad  k=0, \dots, d-1. 
$$
This system of linear equations always has a solution (since ${^\phi}\F_\fp$ has exponent). 
Find a particular solution $\mathbf{x}$ and find a basis 
$\mathbf{b}_1, \dots, \mathbf{b}_h$ for the intersection of null-spaces of all $M_k$, so that every other solution 
is of the form $\mathbf{x}+\mathrm{span}(\mathbf{b}_1, \dots, \mathbf{b}_h)$. 
It is easy to see that the exponent of ${^\phi}\F_\fp$ 
is the gcd of $f_{\mathbf{x}}, f_{\mathbf{x}+\mathbf{b}_1}, \dots,  f_{\mathbf{x}+\mathbf{b}_h}$, 
where $f_{\mathbf{y}}:=y_0+\dots+y_{d-1}T^{d-1}+T^d$ for $\mathbf{y}=(y_0, \dots, y_{d-1})$. 
This can be easily computed. Computing all $\phi_{T^i}(\theta^k)$ can be done in polynomial time in $d$, 
solving the system of linear equations also can be done in polynomial time in $d$. Hence we can find the exponent 
$\fd_r$ of ${^\phi}\F_\fp$ in polynomial time in $d$.

\begin{example} Suppose we want to compute all $\fd_1, \dots, \fd_r$. 
	The previous algorithm allows us to computed $\fd_r$. Since $\fd_1\mid \fd_2\mid \cdots\mid \fd_r$, this already 
	gives us only finitely many possibilities for these invariants. To further restrict the possibilities, one can compute 
	$P(x)$, which then allows to compute the product $\fd_1\cdot \fd_2\cdots \fd_r=P(1)A=\chi(\phi)$, thanks to 
	Proposition \ref{prop1.7}. One can also uniquely determine $\fd_1$ using Lemma \ref{lem1.5}. In practice, 
	knowing $\fd_1, \fd_r$, and the product $\prod_{i=1}^r\fd_r$ is usually sufficient to uniquely determine all $\fd_i$'s. 
	(Obviously, this is always the case when $r\leq 3$.) 
	When this is not sufficient, one can determine these invariants by computing the dimension of the null space of possible $\phi_{\fd_i}$ 
	by an argument used in the algorithm for computing $\fd_r$.   
	
	In this example we take $q=3$ and compute $\fd_1, \fd_2, \fd_3$ for  the Drinfeld $A$-module 
	$$
	\phi_T=\theta+\theta\tau+\tau^3 
	$$ 
over $\F_\fp$ for varying primes $\fp$, which are chosen specifically to demonstrate different possible situations. 

First, let 
$\fp=T^7 + T^5 + T -1$. Then the exponent is $\fd_3=T^3(T+1)(T-1)$. We also have 
$P(X)=X^3-X^2-(T^3 -T +1)X-\fp$.  
Hence $\chi(\phi)=T^3(T+1)^2(T-1)^2$. This implies $\fd_2=(T+1)(T-1)$ and $\fd_1=1$. 

Next, let $\fp=T^8 + T^7 + T^6 + T^4 - T^3 - T^2 -1$. Then 
$\fd_3=(T+1)^2(T-1)^3$ and $\chi(\phi)=(T+1)^3(T-1)^5$. Hence either 
$$
\fd_2=(T+1)(T-1)^2, \quad \fd_1=1,
$$
or 
$$
\fd_2=(T+1)(T-1), \quad \fd_1=(T-1). 
$$
If $\fd_1$ is not $1$, then, by Lemma \ref{lem1.5}, $\phi_{T-1}(x)=(T-1)x+Tx^3+x^{27}$ splits completely modulo $\fp$. This is easy to 
check on a computer to be false, hence $\fd_2=(T+1)(T-1)^2$ and $\fd_1=1$. 

\begin{comment}
\begin{example}
Let $q=3$ and $\phi_T=T+T\tau+\tau^3$. Let $\fp=T^{12} + T^9 -T^8 + T^7 + T^5 + T^4 + T^3 -T^2 -T -1$. Then 
$\fd_3=(T + 1)^2(T^2 + 1)(T^2 + T + 2)^2$ and $\chi(\phi)=(T+1)^4(T^2 + 1)^2(T^2 + T + 2)^2$. 
Hence either 
$$
\fd_2=(T+1)^2(T^2 + 1), \quad \fd_1=1,
$$
or 
$$
\fd_2=(T+1)(T^2 + 1), \quad \fd_1=(T+1). 
$$
Since $\phi_{T+1}(x)$ does not split completely modulo $\fp$, we have $\fd_1=1$. 
\end{example}

\begin{example}
Let $q=3$ and $\phi_T=T+T\tau+\tau^3$. To find an example where $\fd_1\neq 1$, we first make 
a search of primes which split completely in $F(\phi[T])$. We find 
$$
\fp:=T^{11} + T^8 - T^7 - T^6 - T^5 + T^2 + T + 1. 
$$
We compute directly that 
$$
\fd_3=T^9 + T^6 -T^5 + T^3 -T^2 + T=T(T^2 + 1)(T^6 - T^4 + T^3 - T + 1).
$$ 
Hence $\fd_1=\fd_2=T$ and $\chi(\phi)=T^2\fd_3=T^3(T^2 + 1)(T^6 -T^4 + T^3 -T + 1)$. 
This last statement about $\chi(\phi)$ can be verified directly. 
\end{example}
\end{comment}

Finally, let 
$
\fp=T^{14} + T^{13} + T^{12} + T^5 -T^2 + T + 1$. 
Then 
$$
\fd_3 =T(T + 1)(T -1)(T^2 + 1)^2(T^4 - T -1). 
$$
and 
$$
\chi(\phi)=T^3(T + 1)^2(T -1)(T^2 + 1)^2(T^4 - T -1). 
$$
One checks that $\phi[T]$ is rational over $\F_\fp$. Hence we must have 
$$
\fd_1=T, \qquad \fd_2=T(T+1). 
$$
\end{example}

\subsection{Endomorphism ring}\label{ssER}
%We describe an algorithm for computing $\End_{\F_\fp}(\phi)$ assuming  
%the rank of $\phi$ is $2$. In comparison to the known algorithms for computing 
%the endomorphism rings of elliptic curves over finite fields, cf. \cite{DT}, our algorithm is quite 
%different and more elementary. The difference stems from the fact that we crucially use the fact that 
%$\End_{\F_\fp}(\phi)$ is a subring of the larger ambient space $\F_\fp\{\tau\}$. 

Now assume $r=2$. Let 
$$
P(X)=X^2-aX+\eps(\phi)\fp,
$$
be the characteristic polynomial of the Frobenius, so $a$ is the trace of $\rho_{\phi, \fl}(\Frob_\fp)$. %; see Proposition \ref{prop1.8}. 
Let $\pi$ be a root of $P(X)$. Then $K=F(\pi)$ is 
an imaginary quadratic extension of $F$. Let $\cO_K$ be the integral closure of $A$ in 
$K$.  Denote $\cE_\phi:=\End_{\F_\fp}(\phi)$. By Proposition \ref{thm1}, we have the inclusion of orders, 
$$
A[\pi]\subset \cE_\phi\subset \cO_K.  
$$
Let $c_\pi$ (resp. $c_\phi$) be the index of $A[\pi]$ (resp. $\cE_\phi$) in $\cO_K$. These are monic 
polynomials in $A$ such that $c_\phi$ divides $c_\pi$; note that $b=c_\pi/c_\phi$ is the (refined) index of $A[\pi]$ in $\cE_\phi$ 
that appears in Theorem \ref{thm12}. 
Orders in quadratic extensions are uniquely determined by their indices: 
$$
A[\pi]=A+c_\pi\cO_K,\qquad \cE_\phi=A+c_\phi\cO_K, 
$$
so to determine $\cE_\phi$ it is enough to determine $c_\phi$. 

We have  
$\cO_K=A[\alpha]$ for some $\alpha$ satisfying a monic quadratic polynomial $f(X)\in A[X]$. Note that 
$A[c_\pi \alpha]=A[\pi]$, so $c_\pi \alpha=m+n\pi$, where $m\in A$ and $n\in \F_q^\times$. 
Suppose we are able to do the following:
\begin{enumerate}
\item[(i)] Compute $f(X)$. 
\item[(ii)] Compute $c_\pi$.
\item[(iii)] Compute $m, n$ such that $c_\pi \alpha=m+n\pi$. 
\end{enumerate}
Then $c_\phi$ can be computed using the following process: 
Initially, put $c_1=c_\pi$. Let $c\neq 1$ run through monic divisors of $c_1$. For a given $c$ 
we look for $x\in \F_\fp\{\tau\}$ such that  
\begin{equation}\label{eqCentral}
x\phi_c=\phi_m+n \tau^d. 
\end{equation}
If we write $x=x_0+x_1\tau+\cdots+x_s \tau^s$, where $s=\deg_\tau(n\tau^d+\phi_m)-2\deg(c)$, then \eqref{eqCentral} 
gives a system of \textit{linear} equations in $x_0, \dots, x_m$, so can be easily solved. Of course, this system of 
linear equations might not have any solutions, but when it does, the solution $x$ is unique. For such a 
solution we check whether $x\phi_T=\phi_Tx$. If this condition holds (so $x\in \cE_\phi$) 
then we replace $c_1$ 
by $c_1/c$ and repeat the process. Eventually, we will either end up with $c_1=1$ or will not find any $x$ satisfying the 
necessary conditions. In that case, 
the process terminates and the index of $\cE_\phi$ is $c_\phi=c_1$. 
Note that this process also computes a generator of $\cE_\phi$ over $A$. Indeed, it is easy to see that $\cE_\phi=A[x]$, 
where $x$ is the solution of $x\phi_{c_\pi/c_\phi}=n\tau^d+\phi_m$. 

Now we address the question of how to carry out (i)-(iii). There are three cases, which need to be treated 
separately: 

\boxed{$\text{Case 1: $q$ is odd}.$} Let 
$
\Delta_\pi:= a^2- 4\eps(\phi)\fp$. 
Note that $\Delta_\pi\in A$ has degree $\leq d$.  
We can decompose any polynomial $h(T)\in A$ as $h(T)=c^2e$, where $c$ is monic and $e$ is square-free. 
Decompose $\Delta_\pi$ in this manner  
$$
\Delta_\pi:= c^2\cdot \Delta_{\max}. 
$$
Then $f(X)=X^2-\Delta_{\max}$ and $c_\pi=c$; this gives (i) and (ii). If we fix a root $\alpha$ of $f(X)$, then (iii) follows from the 
quadratic formula: $2\pi=a+c_\pi\alpha$. 

\begin{comment}
\begin{rem}
It might appear that for $q$ odd instead of (iii) and the subsequent process one could try to solve $x^2=\phi_{c^2\Delta_{\max}}$ in 
$\F_\fp\{\tau\}$ for each divisor $c$ of $c_\pi$ to determine $c_\phi$, by checking if a solution $x$ commutes with $\phi_T$. 
But quadratic equations in $\F_\fp\{\tau\}$ 
usually have more than two solutions (they essentially correspond to different embeddings of quadratic extensions of $F$ 
into the division algebra $\F_\fp\{\tau\}$), and trying to find all solutions of $x^2=\phi_{c^2\Delta_{\max}}$ leads to 
a system of non-linear equations over $\F_\fp$, which is hard to solve in practice. 
\end{rem}
\end{comment}

\boxed{$\text{Case 2: $q$ is even and $a=0$}.$} This is equivalent to $K/F$ being inseparable. Let $g:=\eps(\phi) \fp$. 
Then $K$ is defined by the equation $X^2=g$. The polynomial $g$ decomposes (uniquely) 
as $g=g_e+g_o$, where $g_e$ (resp. $g_o$) is a polynomial in $T$ whose terms all have even degrees (resp. odd degrees). 
Then $g_e=s^2$ and $g_o=Tc^2$ for uniquely determined $s, c\in A$. After a change of variables $X\mapsto X+s$, 
we see that $A[\sqrt{g}]=A[c\sqrt{T}]$. Hence $\cO_K=A[\sqrt{T}]$, $f(X)=X^2+T$, $c_\pi=c$, and 
$\pi+s=c_\pi\sqrt{T}$. 

\vspace{0.1in}

\boxed{$\text{Case 3: $q$ is even and $a\neq 0$}.$} 
This is the most complicated case; it is equivalent to $K/F$ being separable in characteristic $2$. 
%Let as before $P(X)=X^2-aX+\eps(\phi)\fp$ be the characteristic polynomial of the Frobenius. 
By \cite[III.6, Cor.1]{SerreLF}, 
we have an equality of ideals
$$
(P'(\pi))=(a)=(\cD_K\cdot c_\pi), 
$$
where $\cD_K$ is the different of $\cO_K$ over $A$. Thus, to compute $c_\pi$ we need to compute $\cD_K$. To do this, 
we first recall some facts about Artin-Schreier extensions. 

Any quadratic separable extension $K/F$ is the splitting field of a polynomial  
$$
X^2+X=\fn/\fm
$$
for some $\fn,\fm\in A$ coprime to each other. Suppose $\fm=\fq^{2e}\fm_1$, where $e\geq 1$, $\fq$ is a prime, and $\fq\nmid \fm_1$.  
After a change of variables, $X\mapsto X+b/\fp^e$, $b\in A$, we get 
$$
X^2+X=\frac{\fn+b^2\fm_1+\fq^e\fm_1 b}{\fm}.  
$$
Since $\fn$ and $\fm_1$ are coprime to $\fq$, and squaring is an automorphism of $A/\fq$, we can choose $b$ 
such that $\fn+b^2\fm_1+\fq^e\fm_1 b$ is divisible by $\fq$. Repeating this process finitely many times, 
we can assume that 
$$
\fm=\fq_1^{2e_1-1}\cdots \fq_s^{2e_s-1}, \quad e_1,\dots, e_s\geq 1. 
$$
Then, by \cite[Cor. 2.3]{AST}, 
$$
\cD_K=\fq_1^{e_1}\cdots \fq_s^{e_s}
$$
After a change of variables $X\mapsto X/\epsilon\cD_K$, 
we can rewrite $X^2+X=\fn/\fm=\fq_1\cdots \fq_s \fn/\cD_K^2$ as 
$$
X^2+\epsilon\cD_K X=\epsilon^2\fq_1\cdots \fq_s \fn,
$$
where $\epsilon\in \F_q^\times$ is arbitrary. 
\begin{lem} Let $\alpha$ be a root of  
$f(X)=X^2+\epsilon \cD_K X+\epsilon^2\fq_1\cdots \fq_s \fn$. Then $\cO_K=A[\alpha]$. 
\end{lem}
\begin{proof}
It is clear from the construction that $K=F(\alpha)$. Since $f(X)\in A[X]$ is monic, 
we have $A[\alpha]\subset \cO_K$. 
On the other hand, $f'(X)=\epsilon \cD_K$, so $A[\alpha]= \cO_K$ by \cite[III.6, Cor. 2]{SerreLF}. 
\end{proof}

To compute the different of the extension defined by the characteristic polynomial of the Frobenius $X^2+aX=\eps(\phi)\fp$, 
we first make a change of variables $X\mapsto aX$, then divide both sides by $a^2$, obtaining $X^2+X=\eps(\phi)\fp/a^2$.  
Then we apply the process of the previous paragraph. Let $\epsilon$ be the leading coefficient of $a$ as a polynomial in $T$. 
We have computed 
the minimal polynomial $f(X)=X^2+\epsilon \cD_K X+b$ of an element $\alpha$ generating $\cO_K$. Then $c_\pi \alpha$ is a root of 
$X^2+aX+c_\pi^2b=0$. This implies that $c_\pi\alpha=m+\pi$, where $m\in A$ is such that 
$$
m^2+am=c_\pi^2b+\eps(\phi)\fp. 
$$
If we write $m=m_0+m_1T+\cdots$ as a polynomial in $T$ whose coefficients $m_i$ are unknowns, then the previous equality 
reduces to solving a system of quadratic equations over $\F_q$, which can be done recursively 
starting with the constant term $m_0$. In fact, the first non-zero coefficient $m_h$ of $m$ is determined from a quadratic equation over $\F_q$, while 
every other $m_i$, $i\geq h$, can be deduced from a linear equation in $m_i$ with coefficients involving $m_h, \dots, m_{i-1}$.    

\begin{example}
Let $q=3$ and $\phi$ be a Drinfeld $A$-module of rank $2$ over $F$ given by $$\phi_T=T+g_1\tau+g_2\tau^2.$$ 
Tables \ref{T1} and \ref{1T} list the invariants of $\phi\otimes \F_\fp$ for primes $\fp$ of degree $6$ 
in cases when $A[\pi_\fp]\neq \cE_\pi$. %; note that in notation of this section this condition corresponds to $c_{\pi_\fp}\neq c_{\phi\otimes\F_\fp}$. 
The total number of primes of degree $6$ in $\F_3[T]$ 
is 116, while the number of primes in Table \ref{T1} (resp. Table \ref{1T}) is 24 (resp. 20). 
\end{example}

\begin{example}
 $\F_4$ is generated over $\F_2$ by $w$ satisfying $w^2+w+1=0$. Table \ref{Teven} lists computational data for 
 $$\phi_T=T+T\tau+\tau^2,$$ 
 which involves Cases 2 and 3 of the algorithm computing $\cE_\phi$. 
\end{example}

\begin{table}
\scalebox{0.8}{
\begin{tabular}{ | c || c | c | c | c | c | }
\hline
$\fp$ & $a$ & $\eps(\phi)$ & $c_\pi$ & $c_\phi$ & $\Delta_{\max}$ \\ 
\hline\hline
$T^6 + 2T^5 + 2T^3 + T^2 + 2T + 2$ & $2T^3 + T^2 + 2$ & $1$ & $T + 2$ & $1$ & $2T^3 + 2T^2 + 2T + 2$ \\ \hline 
$T^6 + T^5 + T^4 + 1$ & $2T^3 + 2T^2 + 2T + 2$ & $1$ & $T + 2$ & $1$ & $T^3 + T^2 + 2T$ \\ \hline 
$T^6 + 2T^5 + 2T^4 + 2T^3 + 2T^2 + 2T + 2$ & $2T^3$ & $1$ & $T^2 + 2$ & $T + 1$ & $T + 1$  \\ \hline 
$T^6 + T^5 + 2T^4 + 2T^2 + 2T + 2$ & $2T^3 + T + 2$ & $1$ & $T^2 + 2$ & $T + 2$ & $2T + 2$ \\ \hline 
$T^6 + T^5 + T^4 + T^3 + T + 2$ & $2T^3 + 2T^2 + 1$ & $1$ & $T + 2$ & $1$ & $T^3 + 2T^2 + 2$  \\ \hline 
$T^6 + T^5 + T^4 + 2T^2 + 2$ & $2T^3 + T + 2$ & $1$ & $T + 2$ & $1$ & $2T^3 + T^2 + 2T + 2$  \\ \hline 
$T^6 + T^5 + T^3 + T^2 + T + 2$ & $2T^3 + T + 2$ & $1$ & $T + 1$ & $1$ & $2T^3 + 2T + 2$ \\ \hline 
$T^6 + T^5 + T^4 + T^3 + 2T^2 + 2T + 2$ & $2T^3$ & $1$ & $T^2 + T + 1$ & $T + 2$ & $2T + 1$  \\ \hline 
$T^6 + 2T^4 + T^2 + T + 2$ & $2T^3 + T^2 + 2$ & $1$ & $T + 2$ & $1$ & $T^3 + T^2 + 2$  \\ \hline 
$T^6 + T^5 + 2T^4 + 2T^3 + 2T^2 + 2$ & $2T^3 + 2T + 1$ & $1$ & $T + 2$ & $1$ & $2T^3 + T^2 + 2T + 2$  \\ \hline 
$T^6 + T^4 + T^3 + 2T + 2$ & $2T^3 + T^2 + T + 1$ & $1$ & $T + 2$ & $1$ & $T^3 + T + 2$  \\ \hline 
$T^6 + 2T^4 + 2T^3 + T + 1$ & $2T^3 + T^2 + T + 1$ & $1$ & $T^2 + 2$ & $T + 1$ & $T$  \\ \hline 
$T^6 + 2T^4 + T^3 + T^2 + 2$ & $2T^3 + T^2 + 2T$ & $1$ & $T + 2$ & $1$ & $T^3 + 2T + 1$ \\ \hline 
$T^6 + T^3 + T^2 + 1$ & $2T^3 + 2T + 1$ & $1$ & $T + 2$ & $1$ & $2T^2 + T$  \\ \hline 
$T^6 + 2T^3 + 2T + 2$ & $2T^3 + T^2 + 2$ & $1$ & $T + 2$ & $1$ & $T^3 + 2T + 2$  \\ \hline 
$T^6 + T^3 + 2T^2 + 2T + 1$ & $2T^3 + T^2 + T + 1$ & $1$ & $T^2 + 2T$ & $T + 2$ & $T + 1$  \\ \hline 
$T^6 + T^5 + 2T^4 + T^3 + T^2 + 2$ & $2T^3 + 2T + 2$ & $1$ & $T + 1$ & $1$ & $2T^3 + 2T^2 + T + 2$  \\ \hline 
$T^6 + T^5 + 2T^3 + 1$ & $2T^3 + 1$ & $1$ & $T$ & $1$ & $2T^3 + 2T$  \\ \hline 
$T^6 + 2T^3 + 2T^2 + T + 1$ & $2T^3 + T^2 + 2$ & $1$ & $T + 2$ & $1$ & $T^3 + 2T$  \\ \hline 
$T^6 + 2T^5 + T^4 + 2T + 1$ & $2T^3 + T^2 + 2T$ & $1$ & $T + 2$ & $1$ & $2T^3 + 2T + 2$  \\ \hline 
$T^6 + T^2 + 2T + 1$ & $2T^3 + 2T + 2$ & $1$ & $T$ & $1$ & $2T^2 + 2T$  \\ \hline 
$T^6 + 2T^5 + 2T + 2$ & $2T^3$ & $1$ & $T + 2$ & $1$ & $T^3 + 2T^2 + 1$  \\ \hline 
$T^6 + 2T^5 + T^2 + 2T + 1$ & $2T^3 + T^2 + 2$ & $1$ & $T + 2$ & $1$ & $2T^3 + 2T^2 + T$  \\ \hline 
$T^6 + 2T^5 + 2T^4 + T^3 + 1$ & $2T^3 + T^2 + 2$ & $1$ & $T^2 + T$ & $T$ & $2T + 1$  \\ \hline 
\end{tabular}
}
\caption{$q=3$, $g_1=T$, $g_2=1$}
\label{T1}
\end{table}

\begin{table}
\scalebox{0.76}{
\begin{tabular}{ | c || c | c | c | c | c | }
\hline
$\fp$ & $a$ & $\eps(\phi)$ & $c_\pi$ & $c_\phi$ & $\Delta_{\max}$ \\ 
\hline\hline
$T^6 + 2T^5 + 2T^4 + T^3 + T^2 + T + 2$ & $2T^3 + T + 1$ & $2$ & $T^3 + 2T^2 + T$ & $T^2 + 2T + 1$ & $2$  \\ \hline 
$T^6 + 2T^5 + 2T^4 + T^2 + T + 1$ & $T^3 + T + 1$ & $1$ & $T^2 + 1$ & $1$ & $T$ \\ \hline 
$T^6 + T^5 + 2$ & $2T^3 + 2T^2 + 2$ & $2$ & $T$ & $1$ & $2T^4 + T^2 + 2T + 2$  \\ \hline 
$T^6 + 2T^5 + 2$ & $T^3 + 2T^2 + 2$ & $2$ & $T$ & $1$ & $2T^4 + T^2 + T + 2$  \\ \hline 
$T^6 + 2T^5 + T^4 + T^2 + T + 2$ & $T^3 + T$ & $2$ & $T + 2$ & $1$ & $2T^4 + T^2 + 2T + 2$  \\ \hline 
$T^6 + 2T^5 + 2T^4 + 2T^3 + 2T + 2$ & $T^3 + 2T^2 + 1$ & $2$ & $T + 1$ & $1$ & $2T^4 + 2T^3 + 2T$  \\ \hline 
$T^6 + 2T^4 + 1$ & $T^2 + 1$ & $1$ & $T^3 + 2T$ & $T^2 + 2$ & $2$  \\ \hline 
$T^6 + T^5 + 2T^4 + T^2 + 2T + 1$ & $2T^3 + 2T + 1$ & $1$ & $T^2 + 1$ & $1$ & $2T$  \\ \hline 
$T^6 + 2T^4 + T^2 + T + 2$ & $T^3 + T^2 + 2T + 2$ & $2$ & $T$ & $1$ & $2T^4 + 2T^3 + T^2 + 2T$  \\ \hline 
$T^6 + T^5 + T^4 + T^2 + 2T + 2$ & $2T^3 + 2T$ & $2$ & $T + 1$ & $1$ & $2T^4 + T^2 + T + 2$  \\ \hline 
$T^6 + T^5 + 2T^4 + T^3 + T + 2$ & $2T^3 + 2T^2 + 1$ & $2$ & $T + 2$ & $1$ & $2T^4 + T^3 + T$  \\ \hline 
$T^6 + T^3 + T^2 + 2T + 2$ & $T^3 + 2T^2 + 2T + 1$ & $2$ & $T$ & $1$ & $2T^4 + T^3 + 2T^2 + 2T$  \\ \hline 
$T^6 + 2T^2 + 1$ & $T^2 + 2$ & $1$ & $T^3 + T$ & $T$ & $2$  \\ \hline 
$T^6 + T^3 + 2T^2 + 2T + 1$ & $T^3 + 2T^2$ & $1$ & $T + 1$ & $1$ & $T^3 + 2T^2 + 2$  \\ \hline 
$T^6 + 2T^4 + 2T^3 + T^2 + 2T + 2$ & $2T^3 + 2T^2 + 2T$ & $2$ & $T^2 + 2T + 1$ & $T + 1$ & $2T^2 + 2$  \\ \hline 
$T^6 + T^5 + 2T^4 + 2T^3 + T^2 + 2T + 2$ & $T^3 + 2T + 1$ & $2$ & $T^3 + T^2 + T$ & $T^2 + T + 1$ & $2$  \\ \hline 
$T^6 + 2T^3 + T^2 + T + 2$ & $2T^3 + 2T^2 + T + 1$ & $2$ & $T$ & $1$ & $2T^4 + 2T^3 + 2T^2 + T$  \\ \hline 
$T^6 + 2T^3 + 2T^2 + T + 1$ & $2T^3 + 2T^2$ & $1$ & $T + 2$ & $1$ & $2T^3 + 2T^2 + 2$  \\ \hline 
$T^6 + 2T^4 + T^2 + 2T + 2$ & $2T^3 + T^2 + T + 2$ & $2$ & $T$ & $1$ & $2T^4 + T^3 + T^2 + T$  \\ \hline 
$T^6 + 2T^4 + T^3 + T^2 + T + 2$ & $T^3 + 2T^2 + T$ & $2$ & $T^2 + T + 1$ & $T + 2$ & $2T^2 + 2$ \\ \hline 
\end{tabular}
}
\caption{$q=3$, $g_1=1$, $g_2=T$}
\label{1T}
\end{table}

\begin{table}
	\scalebox{0.76}{
	\begin{tabular}{ |p{7cm}||p{4cm}|p{.7cm}|p{2.5cm}|p{1.5cm}|p{3cm}|  }
		\hline
		$\mathfrak{p}$ & $a$ & $\eps(\phi)$ & $c_{\pi}$ & $c_{\phi}$&  $\cD_K$\\
		\hline
		$T^5 + wT^2 + w^2T + w$ & $T + w$ & 1 & $T + w $ & 1 & 1\\
		\hline
		$T^5 + wT^4 + w^2T^3 + w^2T^2 + wT + w^2$ & $wT^2 + T$ & 1& $T$ & 1 & $ T + w^2$\\
		\hline
		$T^6 + T^5 + w^2T^4 + w^2T^3 + T^2 + w^2T + w^2$ & $w^2T^2 + T + w$ & $1$ & $T^2 + wT + w^2$ & $T + w^2$ & 1\\
		\hline
		$T^7 + T^6 + T^5 + wT^4 + wT + w$ & $T^3 + wT^2 + wT + w^2$ & 1 & $T + w^2$ & 1 & $T^2 + T + 1$\\
		\hline
		$T^7 + T^5 + w^2T^4 + w^2T^2 + 1$ & $0$ & 1 & $T^3 + T^2$ & $T^2$ & $\sqrt{T}$\\
		\hline
		$T^7 + w^2T^5 + T^4 + T^3 + wT^2 + w^2T + w$ & $wT^2 + w$ & 1 & $T^2+1$ & $T+1$ & 1\\
		\hline
		$T^8 + T^6 + wT^5 + w^2T^4 + wT^2 + T + w^2$   & 0   & 1 &  $T^2 + w$ & $T + w+1$ & $\sqrt{T}$\\
		\hline
		$T^9 + T^8 + wT^5 + w^2T^4 + wT^3 + w$ & $T^4 + T^3 + T^2 + 1$ & 1& $T+1$ & 1 & $T^3 + T + 1$\\
		\hline
		$T^{11} + wT^{10} + w^2T^9 + wT^8 + w^2T^7 + wT^4 + wT^3 + wT^2 + T + w^2$ & $wT^5 + wT^4 + T^2 + 1$ & 1 & $T+1$ & 1 &
		$T^4 + w^2T + w^2$ \\
		\hline
	\end{tabular}
}
\caption{$q=4$, $g_1=T$, $g_2=1$}
\label{Teven}
\end{table}

\begin{example} Let $q=3$ and $\phi_T=T+\tau+T\tau^2$. 
We know from Theorem \ref{thm11} that for any fixed prime $\fq$ there exist infinitely many $\fp$ 
such that $c_{\phi\otimes \F_\fp}=\chi(\cO_{\fp, K}/\cE_{\phi\otimes \F_\fp})$ is divisible by $\fq$. We compute that: 

If $\fq=T^2 -T -1$, then the smallest degree $\fp$ is $T^6 + T^5 + T^3 -1$. 

If $\fq=T^3 -T + 1$, then the smallest degree $\fp$ is 
$T^6 + T^4 + T^3 + T^2 -T -1$. 

If $\fq=T^4 -T^3 -1$, then the smallest degree $\fp$ is
$T^{10} + T^9 -T^7 + T^5 -T + 1$. 

By the same corollary, we can also fix two primes $\fq_1$ and $\fq_2$ and find infinitely many $\fp$ such 
that $\fq_1$ divides $c_{\pi_\fp}/c_{\phi\otimes\F_\fp}=\chi(\cE_{\phi\otimes\F_\fp}/A[\pi_\fp])$ 
and $\fq_2$ divides $c_{\phi\otimes\F_\fp}=\chi(\cO_{\fp, K}/\cE_{\phi\otimes \F_\fp})$. 
If $\fq_1=T$ and $\fq_2=T^2 -T -1$, then such a prime of smallest degree is $\fp=T^7 -T^5 -T^4 -1$. 
On the other hand, if $\fq_1=T^2 -T -1$ and $\fq_2=T$, then the smallest degree $\fp$ is 
$T^9 + T^5 + T^4 + T^2 -T + 1$. 
\end{example}

% ----------------------------------------------------------------------

%\clearpage

%\renewcommand{\bibliofont}{\normalsize}
%\bibliographystyle{amsplain}
%\bibliography{EndomDM.bib}

\begin{thebibliography}{10}
	
	\bibitem{AST}
	N.~Anbar, H.~Stichtenoth, and S.~Tutdere, \emph{On ramification in the
		compositum of function fields}, Bull. Braz. Math. Soc. (N.S.) \textbf{40}
	(2009), no.~4, 539--552.
	
	\bibitem{Angles}
	B.~Angl\`es, \emph{On some subrings of {O}re polynomials connected with finite
		{D}rinfeld modules}, J. Algebra \textbf{181} (1996), no.~2, 507--522.
	
	\bibitem{Centeleghe}
	T.~Centeleghe, \emph{Integral {T}ate modules and splitting of primes in torsion
		fields of elliptic curves}, Int. J. Number Theory \textbf{12} (2016), no.~1,
	237--248.
	
	\bibitem{CP}
	A.~Cojocaru and M.~Papikian, \emph{Drinfeld modules, {F}robenius endomorphisms,
		and {CM}-liftings}, Int. Math. Res. Not. IMRN (2015), no.~17, 7787--7825.
	
	\bibitem{CS1}
	A.~Cojocaru and A.~Shulman, \emph{An average {C}hebotarev density theorem for
		generic rank 2 {D}rinfeld modules with complex multiplication}, J. Number
	Theory \textbf{133} (2013), no.~3, 897--914.
	
	\bibitem{CS2}
	A.~Cojocaru and A.~Shulman, \emph{The distribution of the first elementary divisor of the
		reductions of a generic {D}rinfeld module of arbitrary rank}, Canad. J. Math.
	\textbf{67} (2015), no.~6, 1326--1357.
	
	\bibitem{Drinfeld}
	V.~Drinfeld, \emph{Elliptic modules}, Mat. Sb. (N.S.) \textbf{94} (1974),
	594--627.
	
	\bibitem{Drinfeld2}
	V.~Drinfeld, \emph{Elliptic modules. {II}}, Mat. Sb. (N.S.) \textbf{102} (1977),
	182--194.
	
	\bibitem{DT}
	W.~Duke and \'A. T\'oth, \emph{The splitting of primes in division fields of
		elliptic curves}, Experiment. Math. \textbf{11} (2002), no.~4, 555--565
	(2003).
	
	\bibitem{GekelerCR}
	E.-U. Gekeler, \emph{Sur les classes d'id\'eaux des ordres de certains corps
		gauches}, C. R. Acad. Sci. Paris S\'er. I Math. \textbf{309} (1989),
	577--580.
	
	\bibitem{GekelerFDM}
	E.-U. Gekeler, \emph{On finite {D}rinfeld modules}, J. Algebra \textbf{141}
	(1991), no.~1, 187--203.
	
	\bibitem{GekelerTAMS}
	E.-U. Gekeler, \emph{Frobenius distributions of {D}rinfeld modules over finite
		fields}, Trans. Amer. Math. Soc. \textbf{360} (2008), no.~4, 1695--1721.
	
	\bibitem{HYu}
	L.-C. Hsia and J.~Yu, \emph{On characteristic polynomials of geometric
		{F}robenius associated to {D}rinfeld modules}, Compositio Math. \textbf{122}
	(2000), no.~3, 261--280.
	
	\bibitem{KL}
	W.~Kuo and Y.-R. Liu, \emph{Cyclicity of finite {D}rinfeld modules}, J. Lond.
	Math. Soc. (2) \textbf{80} (2009), no.~3, 567--584.
	
	\bibitem{LaumonCDV}
	G.~Laumon, \emph{Cohomology of {D}rinfeld modular varieties. {P}art {I}},
	Cambridge Studies in Advanced Mathematics, vol.~41, Cambridge University
	Press, Cambridge, 1996.
	
	\bibitem{Pink}
	R.~Pink, \emph{The {M}umford-{T}ate conjecture for {D}rinfeld-modules}, Publ.
	Res. Inst. Math. Sci. \textbf{33} (1997), no.~3, 393--425.
	
	\bibitem{Schneider}
	P.~Schneider, \emph{{$p$}-adic {L}ie groups}, Grundlehren der Mathematischen
	Wissenschaften [Fundamental Principles of Mathematical Sciences], vol. 344,
	Springer, Heidelberg, 2011.
	
	\bibitem{SchoofEC}
	R.~Schoof, \emph{Elliptic curves over finite fields and the computation of
		square roots mod {$p$}}, Math. Comp. \textbf{44} (1985), no.~170, 483--494.
	
	\bibitem{SerreLF}
	J.-P. Serre, \emph{Local fields}, Graduate Texts in Mathematics, vol.~67,
	Springer-Verlag, New York-Berlin, 1979.
	
	\bibitem{SerreAR}
	J.-P. Serre, \emph{Abelian {$l$}-adic representations and elliptic curves},
	Research Notes in Mathematics, vol.~7, A K Peters, Ltd., Wellesley, MA, 1998,
	With the collaboration of Willem Kuyk and John Labute, Revised reprint of the
	1968 original.
	
	\bibitem{Takahashi}
	T.~Takahashi, \emph{Good reduction of elliptic modules}, J. Math. Soc. Japan
	\textbf{34} (1982), no.~3, 475--487.
	
	\bibitem{JKYu}
	J.-K. Yu, \emph{Isogenies of {D}rinfeld modules over finite fields}, J. Number
	Theory \textbf{54} (1995), no.~1, 161--171.
	
	\bibitem{Zarhin}
	Yu. Zarhin, \emph{Endomorphism rings of reductions of elliptic curves and
		{A}belian varieties}, Algebra i Analiz \textbf{29} (2017), no.~1, 110--144.
	
\end{thebibliography}

\end{document}